\documentclass[12pt,a4paper]{article}

\usepackage{amsfonts}
\usepackage{amsmath}
\usepackage{amsbsy}
\usepackage{theorem}
\usepackage{graphicx}

\newtheorem{definition}{Definition}
\newtheorem{theorem}{Theorem}
\newtheorem{proposition}{Proposition}

\newtheorem{remark}{Remark}

\begin{document}

\title{Noncentral limit results for spatiotemporal random fields  on manifolds and beyond}
\sf
\date{}

 \maketitle

\centerline{\bf  M.D. Ruiz-Medina}
\bigskip
\begin{abstract}
This paper derives noncentral limit theorems (NCLTs)   for suitable scaling of functionals of  spatially homogeneous and isotropic, and stationary in time, LRD Gaussian subordinated Spatiotemporal Random Fields (STRFs) with Hermite rank equal to two. The cases  of connected and compact two point homogeneous spaces  $\mathbb{M}_{d}\subset \mathbb{R}^{d+1},$  and  compact convex sets $\mathcal{K}\subset \mathbb{R}^{d+1}$ whose interior has
positive Lebesgue measure, are analyzed. These NCLTs are  obtained in the second Wiener Chaos by applying reduction theorems. The methodological approaches adopted in the derivation of these results are based on the pure point and continuous spectra of the Gaussian STRFs subordinators defined on  $\mathbb{M}_{d}$ and $\mathcal{K},$ respectively.
\end{abstract}

\maketitle


.

 \section{Introduction}
 It well--known that Rosenblatt probability distribution attracts sums of nonlinear transformations of strong--correlated sequences, and integrals of LRD stochastic processes and RFs. In this paper, we pay attention to the last case,  where the Rosenblatt--type   probability distribution arises, characterizing the limiting behavior of suitable scalings of linear functionals of LRD Hermite STRFs living in the second Wiener Chaos.
  Reduction theorems extrapolate this asymptotic behavior to linear functionals of Gaussian subordinated STRFs with Hermite rank equal to two.
 The existing vast literature on this topic focuses on stationary time--indexed sequences, stochastic processes, and isotropic RFs on spatial domains.
 A discretized version in time of these limit results first appears in \cite{Rosenblatt},
 and the limit functional version is considered in \cite{Taqqu75} in the form of the Rosenblatt process. The spectral theory based on second--order moments is usually applied in  this
setting, when mean--square convergence is proved, leading to the integral representation of the limit random variable in terms of  a double
Wiener-It{\^o} stochastic integral (see \cite{Dobrushin};
\cite{Taqqu79}). Just to mention a few, we refer to \cite{Albin}, \cite{Anh}, \cite{Fox},
\cite{Ivanov}, \cite{LeonenkoTaufer06}, \cite{Rosenblatt79}, being   relevant classical references in this topic. Alternatively, weak convergence can
be  proved, in terms of the characteristic functions of the elements of the functional sequence. This   approach  was initially  adopted in  \cite{Taqqu75} (see  also \cite{LeonenkoTaufer06}  for functions of quadratic forms of
strongly-correlated Gaussian random variables sequences). In \cite{LeonenkoRuizMedinaTaqqu17},   the limit  characteristic function is  obtained from  Fredholm determinant formula via   the Karhunen--Lo\'eve expansion of the spatially homogeneous and isotropic LRD Gaussian RF subordinator  on a  spatial compact convex set.  In \cite{Leonenkoetal17}, this  formulation is extended to the Laguerre Chaos.
In this paper,  the  limit characteristic function is obtained from   Fredholm determinant  formula (see Proposition 1.2.8 in \cite{DaPrato}), involving the positive integer powers of the eigenvalues   of the  restriction of the spatiotemporal Riesz potential to a spatiotemporal compact set.

  The spatiotemporal formulation of the above cited   limit results,  for LRD Gaussian STRF subordinators, has been considered in \cite{LeonenkoRuizMedina23} and
\cite{LeonenkoRuizMedina25}, where spatiotemporal reduction theorems, and CLTs are obtained. Note that,  in \cite{LeonenkoRuizMedina23}, increasing domain asymptotic in time for a fixed compact  domain in space is analyzed.  While, in \cite{LeonenkoRuizMedina25},   increasing domain asymptotic  in time and space is addressed, including  CLTs for   moving level sojourn measures (see also \cite{Caponera25}  and \cite{MarinucciRV}   on limit results for sojourn measures of chi--squared STRFs on compact two point homogeneous spaces, and of Gaussian  STRFs on sphere,  respectively).

The present paper extends to the spatiotemporal context NCLTs derived in \cite{LeonenkoRuizMedinaTaqqu17}  and \cite{Leonenkoetal17}, adopting the methodological approach introduced in \cite{LeonenkoRuizMedina23} and  \cite{LeonenkoRuizMedina25}. Particularly, we consider  the construction of  functionals, via integrals of Gaussian subordinated STRFs with Hermite rank equal to two,  defined on an  increasing spatiotemporal compact domain family. A  homothetic transformation of the spatial domain, with scaling factor  involving a power law of the increasing length of the temporal interval, defines  this family. We pay attention to the cases of spherical STRFs, and STRFs on compact convex sets with positive Lebesgue measure.
In the first case, we compute the limit characteristic function, and exploit the second--order pure point  spectral properties of the family of  restricted Gaussian STRF subordinators. The novelty of the approach presented here relies on the application of  the spatiotemporal Karhunen--Lo\'eve expansion of the Gaussian STRF subordinator restricted to the family of increasing spatiotemporal domains. The scaling factor is then computed from the application of trace formula on these domains, allowing the connection between pure point spectral properties of the restricted covariance operator, and the asymptotic properties  of the original unrestricted second--order model. The new  limit random variable  obeying a  Rosenblatt--type  probability distribution is given by an infinite weighted sum of chi-squared  distributions with  increasing degrees of freedom,  given by the dimension of the eigenspaces of the Laplace--Beltrami operator  $\Delta_{d}$ on the sphere. Indeed, its characteristic function  can be interpreted as an infinite product of Fredholm determinants corresponding to  the temporal evolution of the projected STRF into each eigenspace of the Laplace Beltrami operator. This limit  formula of the characteristic function is given in terms of the  cumulants,  defined  from  the weights involved in the  random series characterizing the Rosenblatt--type  limit random variable.  These weights display  the asymptotic separability in space and time assumed for the original unrestricted Gaussian STRF model.

The spatiotemporal mean--square convergence to  the Rosenblatt--type  probability distribution is also derived. The cases of  compact convex spatial domains with positive and null  Lebesgue measure in $\mathbb{R}^{d+1}$  are  considered.  In the case of
  compact convex sets of positive Lebesgue measure, the continuous spectral properties    of the unrestricted covariance operator of the Gaussian STRF subordinator are exploited. In this case, the limit  Rosenblatt--type distribution admits a spatiotemporal double Wiener-It{\^o} stochastic integral representation.
    In the case of manifolds, like the sphere, the invariance against the group of isometries of the manifold,  and the pure point spectral properties of the restricted spatiotemporal covariance models are applied, via the spatiotemporal Karhunen--Lo\'eve expansion of the underlying  restricted Gaussian STRFs.

The outline of the paper is the following. Section \ref{prel1} introduces notation and some preliminary definitions. In Section \ref{ncltchf}, we prove convergence of the characteristic functions of the spherical STRF based  functionals, on an   increasing spatiotemporal domain family,  to the  characteristic function of a Rosenblatt--type distributed random variable.  Section \ref{ncltspectr} provides the spatiotemporal convergence in  the mean--square sense  to a Rosenblatt--type distribution in the case of compact convex sets with positive Lebesgue measure. In Section \ref{mssphere},
the mean--square  convergence for STRFs on the sphere is addressed. Finally, Section \ref{discussion} summarizes main contributions, and related open problems to be addressed in a near future.

 \section{Preliminaries}
 \label{prel1}
Let  $\mathbb{T}\subset \mathbb{R}$ be a bounded closed interval. As before, $\mathbb{M}_{d}$ denotes  a connected  and compact two--point homogeneous space. Note that $\mathbb{M}_{d}$ constitutes an example of manifold, with isometrically equivalent  properties to the sphere, locally resembles an Euclidean space.
Here, $d$ represents  its topological dimension. In what follows we denote by  $L^{2}\left(\mathbb{M}_{d}\times \mathbb{T},d\nu\otimes dt\right)$  the separable  Hilbert space of square integrable functions on  $\mathbb{M}_{d}\times \mathbb{T},$ with  $dt=|\mathbb{T}|\lambda (dt),$ and being
$|\mathbb{T}|$ the Lebesgue measure of  $\mathbb{T},$ and    $\lambda (dt)$   the uniform measure on temporal  interval $\mathbb{T}.$  Notation $d\nu $ means that we are integrating with respect to the corresponding non--normalized  Riemannian measure.

Let $(\Omega ,\mathcal{A},P)$ be the basic probability space, and denote by $\mathcal{L}^{2}(\Omega ,\mathcal{A},P)$ the space of zero--mean random variables on  $(\Omega ,\mathcal{A},P)$ with finite second--order moments. Consider
$Z:=\{ Z(\mathbf{x},t),\ \mathbf{x}\in \mathbb{R}^{d+1},\ t\in \mathbb{R}\}$ to  be a wide sense stationary in time, and homogeneous and isotropic in space,  zero--mean,  and mean--square continuous Gaussian STRF with covariance kernel $C_{Z},$ and variance one.
In the following,  we denote by
$\left\{\mathbb{M}_{d}(T^{\gamma })\times \mathbb{T}(T),\ T\geq 1,\ \gamma\geq 0\right\}$  a family of spatiotemporal compact sets  with $\mathbb{M}_{d}(T^{\gamma })=T^{\gamma }\mathbb{M}_{d},$ and $\mathbb{T}(T)=[-T,T].$

 In what follows, we work under the next  assumption, providing the diagonal series  representation of the spatiotemporal covariance kernel
\begin{eqnarray}&&C_{\mathbb{M}_{d}(T^{\gamma })\times \mathbb{T}(T)}(\mathbf{x},\mathbf{y},t,s):=E\left[ Z_{\mathbb{M}_{d}(T^{\gamma })\times \mathbb{T}(T)}(\mathbf{x},t)Z_{\mathbb{M}_{d}(T^{\gamma })\times \mathbb{T}(T)}(\mathbf{y},s)\right]\nonumber\\
&&=\mathbb{I}_{\mathbb{M}_{d}(T^{\gamma })\times \mathbb{T}(T)}(\mathbf{x},t)\mathbb{I}_{\mathbb{M}_{d}(T^{\gamma })\times \mathbb{T}(T)}(\mathbf{y},s)C_{Z}(\|\mathbf{x}-\mathbf{y}\|,|t-s|),\nonumber\\ &&\hspace*{4cm} \forall  \mathbf{x},\mathbf{y}\in \mathbb{R}^{d+1},\quad  t,s\in \mathbb{R},\nonumber
\end{eqnarray}
\noindent where $$Z_{\mathbb{M}_{d}(T^{\gamma })\times \mathbb{T}(T)}:=\mathbb{I}_{\mathbb{M}_{d}(T^{\gamma })\times \mathbb{T}(T)}(\mathbf{x}, t)Z(\mathbf{x}, t),\ \mathbf{x}\in \mathbb{R}^{d+1},\ t\in  \mathbb{R},$$\noindent is the restriction of  $Z$  to  $\mathbb{M}_{d}(T^{\gamma })\times \mathbb{T}(T),$ for any $T\geq 1.$ Here, $\mathbb{I}_{\mathbb{M}_{d}(T^{\gamma })\times \mathbb{T}(T)}$ is the indicator function of the set $\mathbb{M}_{d}(T^{\gamma })\times \mathbb{T}(T).$

\medskip

\noindent \textbf{Assumption A0}.
Assume that, for each $T\geq 1,$ the continuous covariance function \linebreak $C_{\mathbb{M}_{d}(T^{\gamma })\times \mathbb{T}(T)}(\mathbf{x},\mathbf{y},t,s)$  of the restriction  $Z_{\mathbb{M}_{d}(T^{\gamma })\times \mathbb{T}(T)}$ to $\mathbb{M}_{d}$ of $Z$
admits the following diagonal series expansion:
 \begin{eqnarray}&&
 \hspace*{-0.7cm}C_{\mathbb{M}_{d}(T^{\gamma })\times \mathbb{T}(T)}(\mathbf{x},\mathbf{y},t,s)=
  \sum_{n\in \mathbb{N}_{0}}B_{n}(0,T^{\gamma })\sum_{k\geq 1}B_{n,k}(T)\nonumber\\
  &&\hspace*{3.5cm}+
  \sum_{j=1}^{\Gamma (n,d)}
S_{n,j}^{d}(\mathbf{x})S_{n,j}^{d}(\mathbf{y})\phi_{k}^{\mathbb{T}(T)}(t)\phi^{\mathbb{T}(T)}_{k}(s),
 \label{klexpc2}
\end{eqnarray}
\noindent for every $t,s \in \mathbb{T}(T),$ and $\mathbf{x}, \mathbf{y}\in \mathbb{M}_{d}(T^{\gamma }),$
with   $\Gamma (n,d)$ denoting  the dimension of the $n$th eigenspace of the Laplace Beltrami operator on $L^{2}(\mathbb{M}_{d},d\nu)$. Along the paper we will denote by   $\left\{S_{n,j}^{d},\ j=1,\dots,\Gamma(n,d),\ n\in \mathbb{N}_{0}\right\}$ and $\left\{ \phi_{k}^{\mathbb{T}(T)},\ k\geq 1\right\}$  the eigenfunctions of the Laplace--Beltrami operator,  and   an orthonormal basis of $L^{2}(\mathbb{T}(T),dt),$ respectively.

\subsection{Examples}
\label{examples}
\textbf{Assumption A0} is satisfied, for example, by the Gneiting class of spatiotemporal  covariance functions  (see \cite{Gneiting02}) restricted to  $\mathbb{S}_{d}(T^{\gamma })\times \mathbb{T}(T),$ with  $\mathbb{S}_{d}(T^{\gamma })=\left\{ \mathbf{x}\in \mathbb{R}^{d+1}; \|\mathbf{x}\|=T^{\gamma }\right\} ,$ $T\geq 1$. Specifically, for the case $d=3,$ and  $T=1,$ we have
  \begin{eqnarray}&&
  \hspace*{-0.7cm} C\left( 2\sin \left( \frac{\theta }{2}\right),\tau \right)=\frac{\sigma ^{2}}{[\psi (\tau ^{2})]^{d/2}}
\varphi \! \left( \frac{\left[2\sin \left( \frac{\theta }{2}\right)\right]^{2}}{\psi (\tau ^{2})}\right),\
\sigma ^{2}\geq 0,\ \theta \in [0,\pi],\ \tau \in  \mathbb{T}(1),\nonumber\\
  \label{rgneiting}
  \end{eqnarray}
\noindent which is  isotropic in space, and stationary in time. Function  $\varphi (u),$ $u\geq 0$ is a completely monotone function, that is, it
possesses derivatives $\varphi ^{\left( n\right) }$ of all orders and they
alter in sign: $\left( -1\right) ^{n}\varphi ^{\left( n\right) }\left(
u\right) \geq 0$ for $u>0$ and $n=0,1,2,...$ We also assume that $\psi (u),u\geq 0,$ is a Bernstein
function, that is, a  positive function with a
completely monotone derivative.

The  invariance of covariance kernel family (\ref{rgneiting}),  under translations in time and rotations in space, allows the diagonalization of their restricted elements  in time and space  in terms of the tensorial product of the bases of  complex exponential  indexed by $\mathbb{Z},$   and    spherical harmonics basis
$\left\{S_{n,j}^{d},\ j=1,\dots,\Gamma(n,d),\ n\in \mathbb{N}_{0}\right\}$.  See, for example,  Chapters 2 and 3 in \cite{Marinucci} on invariance properties, in the context of  locally compact Abelian groups (applied here to time), and compact and non-Abelian Lie  groups (applied here to space), respectively.

Let us pay attention to the following parametric subfamilies of the restricted Gneiting class on the sphere introduced in equation  (\ref{rgneiting}).
\subsubsection*{Example 1} This example corresponds to the case where functions  $\varphi $ and $\psi$  in equation (\ref{rgneiting}) are given by
\begin{eqnarray} \varphi \left(\frac{\left[2\sin \left( \frac{\theta }{2}\right)\right]^{2}}{\psi (\tau ^{2})}\right)&=&\frac{1}{\left(1+c\left(\frac{\left[2\sin \left( \frac{\theta }{2}\right)\right]^{2}}{\left(1+a\tau ^{2\alpha }\right)^{\beta }}\right)^{\widetilde{\gamma }}\right)^{\nu }},\quad  \theta \in [0,\pi],  \ c>0,\ 0<\widetilde{\gamma } \leq 1,\ \nu >0,\nonumber\\
\psi (\tau ^{2})&=&\left(1+a\tau ^{2\alpha }\right)^{\beta },\quad  \tau \in \mathbb{T}(1), \ a>0,\ 0<\alpha \leq 1,\ 0<\beta \leq 1.
\label{A1}
\end{eqnarray}

\noindent This covariance function subfamily displays LRD in time.

\subsubsection*{Example 2}
An equivalent LRD asymptotic behavior in time  is observed when we  consider
\begin{eqnarray}
\varphi  \left(\frac{\left[2\sin \left( \frac{\theta }{2}\right)\right]^{2}}{\psi (\tau ^{2})}\right)&=& E_{\widetilde{\nu}  ,1}\left( -\left(\frac{\left[2\sin \left( \frac{\theta }{2}\right)\right]^{2}}{\left(1+a\tau ^{2\alpha }\right)^{\beta }}\right)^{\widetilde{\gamma }}\right),\quad  \theta \in [0,\pi],\ 0<\widetilde{\gamma }\leq 1,\ 0<\widetilde{\nu} \leq 1,
\nonumber\\
\psi (\tau ^{2})&=&\left(1+a\tau ^{2\alpha }\right)^{\beta },\quad \tau \in \mathbb{T}(1),\ a>0,\ 0<\alpha \leq 1,\ 0<\beta \leq 1.
\end{eqnarray}

\noindent
Here, $E_{\widetilde{\nu} ,1}$ denotes the Mittag--Leffler function, which is a well-known  two-parameter function defined by the following series expansion:
\begin{equation*}
E_{\widetilde{\nu}  ,\varrho }(-u)=-\sum_{k=1}^{N}\frac{(-u)^{-k}\Gamma (\varrho )}{%
\Gamma (\varrho -\widetilde{\nu}  k)}+O\left( \frac{1}{|u|^{N+1}}\right).
\end{equation*}
\noindent We consider the special case
$$E_{\widetilde{\nu}}(-z)=\sum_{k=0}^{\infty}\frac{(-1)^{k}z^{k}}{\Gamma(\widetilde{\nu}  k+1)},\quad z\geq 0$$ \noindent  (see, e.g.,  \cite{Gorenfloetal2014}).
The LRD behavior in time displayed by this subfamily of spatiotemporal covariance functions can be characterized in
terms of the following two-sided estimates:
\begin{eqnarray}
\frac{1}{1+\Gamma (1-\widetilde{\nu} )x} &\leq & E_{\widetilde{\nu} ,1}(-x)\leq \frac{1}{1+[\Gamma (1+\widetilde{\nu} )]^{-1}x},
\label{eqimlf}
\end{eqnarray}
\noindent for $x\in \mathbb{R}_{+},$ and $\widetilde{\nu} \in (0,1)$  (see  \cite{Simon2014}, Theorem 4).

\subsubsection*{Example 3}
Another interesting example is given by the following choice of functions $\varphi$ and  $\psi,$ related to the
Whittle-Mat\'ern covariance family:

\begin{eqnarray}
 \varphi \left(\frac{\left[2\sin \left( \frac{\theta }{2}\right)\right]^{2}}{\psi (\tau ^{2})}\right)&=&
\left( 2^{\nu -1}\Gamma \left( \nu \right) \right) ^{-1}\left(
c\left(\frac{\left[2\sin \left( \frac{\theta }{2}\right)\right]^{2}}{\left(1+a\tau ^{2\alpha }\right)^{\beta }}\right)^{1/2}\right) ^{\nu }
\nonumber\\ &\times & K_{\nu }\left( c\left(\frac{\left[2\sin \left( \frac{\theta }{2}\right)\right]^{2}}{\left(1+a\tau ^{2\alpha }\right)^{\beta }}\right)^{1/2}\right), \quad \theta \in [0,\pi], \ \text{ }c>0,\text{ }%
\nu >0,\nonumber\\
\psi (\tau ^{2})&=&\left(1+a\tau ^{2\alpha }\right)^{\beta },\quad \tau \in \mathbb{T}(1),\ a>0,\text{ }0<\alpha \leq 1,
\text{}  0<\beta \leq 1,\nonumber\\
\label{A2}
\end{eqnarray}
\noindent displaying a fractal behavior  in space, where $K_{\nu }(z)$ is the modified Bessel function of the
second kind of order $\nu.$

 \subsection{Karhunen-Lo\'eve expansion and  Fredholm determinant formula}

  Applying   invariance  under the group of isometries of  $\mathbb{I}_{\mathbb{M}_{d}(T^{\gamma })}(\mathbf{x})\mathbb{I}_{\mathbb{M}_{d}(T^{\gamma })}(\mathbf{y})C_{Z}(\|\mathbf{x}-\mathbf{y}\|,|t-s|),$
\begin{eqnarray}&&
\mathbb{I}_{\mathbb{M}_{d}(T^{\gamma })}(\mathbf{x})\mathbb{I}_{\mathbb{M}_{d}(T^{\gamma })}(\mathbf{y})C_{Z}(\|\mathbf{x}-\mathbf{y}\|,|t-s|)=
\sum_{n\in \mathbb{N}_{0}}B_{n}(t-s, T^{\gamma })\sum_{j=1}^{\Gamma (n,d)}
 S_{nj}^{(d)}(\mathbf{x})S_{n,j}^{d}(\mathbf{y}),\nonumber\\
 \label{igir}\end{eqnarray}
\noindent where, for any $t,s\in \mathbb{R},$  \begin{eqnarray}&&B_{n}(t-s, T^{\gamma })=\nonumber\\ &&\hspace*{-0.7cm}=E\left[ \left\langle Z(\cdot ,t), S_{nj}^{(d)}(\cdot)\right\rangle_{L^{2}(\mathbb{M}_{d}(T^{\gamma })\times \mathbb{T}(T), d\nu\otimes dt)}\left\langle Z(\cdot ,s),S_{nj}^{(d)}(\cdot)\right\rangle_{L^{2}(\mathbb{M}_{d}(T^{\gamma })\times \mathbb{T}(T), d\nu\otimes dt)}\right].\nonumber\end{eqnarray}

The orthonormal basis  $\left\{\phi_{k}^{\mathbb{T}(T)},\ k\geq 1\right\}$  of $L^{2}(\mathbb{T}(T),dt)$ in (\ref{klexpc2}) is such that
\begin{eqnarray}\int_{\mathbb{T}(T)}\widetilde{B}_{n}(t,s)\phi_{k}^{\mathbb{T}(T)}(s)ds &=&\int_{0}^{T}\frac{B_{n}(t-s,T^{\gamma })}{B_{n}(0, T^{\gamma })}\phi^{\mathbb{T}(T)}_{k}(s)ds
\nonumber\\
&=& B_{n,k}(T)\phi_{k}^{\mathbb{T}(T)}(t),\ t\in \mathbb{T}(T),\ k\geq 1,  \ n\geq 0.\label{autoeq}
\end{eqnarray}

From (\ref{igir}), for each finite $T\geq 1,$   the following  spatiotemporal   Karhunen--Lo\'eve expansion holds, with convergence in $L^{2}(\Omega \times\mathbb{M}_{d}( T^{\gamma })\times \mathbb{T}(T),\mathcal{P}\otimes d\nu\otimes dt):$
\begin{eqnarray}\mathbb{I}_{\mathbb{M}_{d}(T^{\gamma })\times \mathbb{T}(T)}(\mathbf{x},t) Z(\mathbf{x},t)&\underset{D}{=}&\sum_{n=0}^{\infty}\sum_{k\geq 1}\sqrt{B_{n}(0, T^{\gamma })B_{n,k}}(T)\nonumber\\
&&\hspace*{0.7cm}+\sum_{j=1}^{\Gamma (n,d)}\eta_{n,j,k} S_{nj}^{(d)}(\mathbf{x})\phi_{k}^{\mathbb{T}(T)}(t),\label{KLexp}\end{eqnarray}
\noindent   where  $\underset{D}{=}$ denotes equality in distribution, which also holds in view of the convergence in the mean--square sense, with $\eta_{n,j,k},$ $j=1,\dots,\Gamma (n,d),$ $n\in \mathbb{N}_{0},$ $k\geq 1,$ being independent and identically distributed standard Gaussian random variables.
\begin{remark}
\label{rem1}
It is important to note that, in this paper, for each $T\geq 1,$  we exploit the  invariance properties of the restriction to $\mathbb{M}_{d}(T^{\gamma })\times \mathbb{T}(T)$ (respectively, to $\mathbb{S}_{d}(T^{\gamma })\times \mathbb{T}(T)$)  of a  homogeneous and isotropic in space, and stationary in time, STRF on $\mathbb{R}^{d+1}\times \mathbb{R}.$ Specifically, this restriction is a subclass of the family of isotropic STRFs on $\mathbb{M}_{d}(T^{\gamma })$ (respectively, on $\mathbb{S}_{d}(T^{\gamma })$), which are stationary in time.
\end{remark}

\begin{remark}
Given the asymptotic separability in space and time  assumed in \textbf{Assumption A1} below, the case of $\left\{\phi_{k}^{\mathbb{T}(T)}(n),\ k\geq 1,\ n\geq 0\right\}$ is not analyzed here, but similar results hold in this more general scenario, under the condition of   biorthonormality of the bases $\left\{\phi_{k}^{\mathbb{T}(T)}(n),\ k\geq 1,\ n\in \mathbb{N}_{0},\right\}$.
\end{remark}

We now introduce the Fredholm determinant of an
operator $A,$ as a complex-valued function which generalizes the
determinant of a matrix.

\begin{definition}
\label{def1} (see, for example, \cite{Simon05},  Chapter 5,  pp. 47-48,
equation (5.12)) Let $A$ be a trace operator on a separable Hilbert
space $H.$ The Fredholm determinant of $A$ is
\begin{equation}
\mathcal{D}(\omega )=\mbox{det}(I-\omega A)=\exp\left(-\sum_{k=1}^{\infty }%
\frac{\mbox{Tr}A^{k}}{k}\omega^{k}\right)=\exp\left(-\sum_{k=1}^{\infty
}\sum_{l=1}^{\infty}|\lambda_{l}(A)|^{k}\frac{\omega^{k}}{k}\right),
\label{fdf}
\end{equation}
\noindent for $\omega \in \mathbb{C},$ where $\left\{\lambda_{l}(A),\ l\geq 1\right\}$ denotes the sequence of singular values of $A,$  and $|\omega |\|A\|_{1}< 1.$ Here, we have applied the identity $$\|A^{k}\|_{1}:=\mbox{Tr}(A^{k})=\sum_{l\geq 1}|\lambda_{l}(A)|^{k},\quad k\geq 1.$$ Note
that $\| A^{m}\|_{1}\leq \|A\|_{1}^{m},$ for $A$ being a trace operator.
\end{definition}

From  equation (\ref{KLexp}),  Parseval identity leads to

\begin{equation}
\int_{\mathbb{T}(T)}\int_{\mathbb{M}_{d}(T^{\gamma })}Z^{2}(\mathbf{x},t)d\nu(\mathbf{x})dt =
\sum_{n=0}^{\infty }\sum_{j=1}^{\Gamma (n,d)}\sum_{k\geq 1}B_{n}(0,T^{\gamma })B_{n,k}(T)\eta _{n,j,k}^{2}.
\label{Ysquare}
\end{equation}

From equation (\ref{Ysquare}), considering Definition \ref{def1},
 we derive the expression of the characteristic function family $\left\{ \varphi_{T} ,\ T\geq 1\right\},$ given by:
\begin{eqnarray}
&& \varphi_{T} (\xi ):=\mathrm{E}\left[ \exp \left( \mathrm{i}\xi \int_{\mathbb{T}(T)}\int_{\mathbb{M}_{d}(T^{\gamma })}Z^{2}(\mathbf{x},t)d\nu(\mathbf{%
x})dt\right) \right] \nonumber\\&&=\prod_{n=0}^{\infty }\prod_{k\geq 1}\left(1-2\mathrm{i}\xi B_{n}(0,T^{\gamma })B_{n,k}(T)\right)^{-\Gamma(n,d)/2}
\nonumber\\
&&=\exp \left( \frac{1}{2}\sum_{m=1}^{\infty }\frac{(2\mathrm{i}\xi )^{m}}{m}
\sum_{n=0}^{\infty }\Gamma(n,d)[B_{n}(0,T^{\gamma })]^{m}\sum_{k\geq 1}[B_{n,k}(T)]^{m}\right), \quad  T\geq 1, \nonumber\\
\label{tvarcf}
\end{eqnarray}
\noindent for $\xi$ such that $$\left|\xi\right|<\frac{1}{2\left[\sum_{n=0}^{\infty }\sum_{k\geq 1}\Gamma(n,d)B_{n}(0,T^{\gamma })B_{n,k}(T)\right]},$$
 \noindent where  the sequences $\left\{B_{n,k}(T),\ k\geq 1,\ n\geq 0\right\}$  and $\left\{B_{n}(0,T^{\gamma }),\ n\geq 0\right\}$ have been introduced in  (\ref{klexpc2}). In the case of the sphere,
$B_{n}(0,T^{\gamma }),$ $n\geq 0,$  are given by
\begin{eqnarray}
&&B_{n}(0,T^{\gamma })=2\Gamma \left( \frac{d}{2}\right)\pi^{d/2}\int_{0}^{\infty}\int_{0}^{\infty}\left[\frac{J_{n+\frac{d-1}{2}}(T^{\gamma }\lambda )}{T^{\gamma }\lambda^{\frac{d-1}{2}}}\right]^{2}\mathcal{G}_{Z}(d\lambda ,d\omega ),\quad  n\in
\mathbb{N}_{0}.\nonumber\\
\label{ccf}
\end{eqnarray}
\noindent  Note that, under spatial  isotropy,  one can consider  the one-dimensional   spectral measure $\mathcal{G}_{Z},$ given by
$$\mathcal{G}(\lambda ,\mu )=\int_{\|\boldsymbol{\xi }\|< \lambda }\int_{|\varrho |<\mu}f_{Z}(\boldsymbol{\xi }, \varrho )d\varrho d\boldsymbol{\xi },\quad \forall  (\lambda ,\mu )\in \mathbb{R}_{+}^{2},$$ \noindent  in terms of
  $$f_{Z}\left(\|\boldsymbol{\xi} \|,\omega \right)=\int_{\mathbb{R}^{d+1}\times \mathbb{R}}\exp(-i\omega \tau)\exp\left(-i\left\langle \boldsymbol{\xi},\mathbf{x} \right\rangle\right)C_{Z}(\mathbf{x},\tau)d\mathbf{x}d\tau$$ \noindent defining the spatiotemporal Fourier transform of the unrestricted spatiotemporal covariance kernel  $C_{Z}$ (see \cite{Ivanov}, for more details).  In equation (\ref{ccf}),  we have denoted by
\begin{equation*}
J_{\theta}(z)=\sum_{m=0}^{\infty}(-1)^{m}\left(\frac{z}{2}\right)^{2m+\theta}%
\left[m!\Gamma (m+\theta +1)\right]^{-1}
\end{equation*}
\noindent  the Bessel function of the first kind, and  order $\theta>1/2,$ involved in the computation of the characteristic function of the uniform probability distribution in the sphere.

In what follows, we will restrict our attention to the sphere.
For $T=1,$ we obtain
\begin{eqnarray}
&&\varphi_{1}(\xi )=\mathrm{E}\left[ \exp \left( \mathrm{i}\xi \int_{\mathbb{T}(1)}\int_{\mathbb{S}_{d}(1)}Z^{2}(\mathbf{x},t)d\nu(\mathbf{%
x}) dt\right)\right] \nonumber\\
&&=\prod_{n=0}^{\infty }\prod_{k\geq 1}\left(1-2\mathrm{i}\xi B_{n}(0,1)B_{n,k}(1)\right)^{-\Gamma(n,d)/2}
\nonumber\\
&&=\prod_{n=0}^{\infty }\exp \left( \frac{\Gamma(n,d)}{2}\sum_{m=1}^{\infty }\frac{(2\mathrm{i}\xi )^{m}}{m}
\sum_{k\geq 1}[B_{n}(0,1)B_{n,k}(1)]^{m}\right) \nonumber\\
&&\hspace*{1cm}=\exp \left( \frac{1}{2} \sum_{m=1}^{\infty }\frac{(2\mathrm{i}\xi )^{m}}{m}\sum_{n=0}^{\infty}\Gamma(n,d)[B_{n}(0,1)]^{m}\sum_{k\geq 1}[B_{n,k}(1)]^{m}\right),
\label{eqfdext}
\end{eqnarray}
\noindent which is finite  under the condition $\left[\sum_{n=0}^{\infty }\sum_{k\geq 1}\Gamma(n,d)B_{n}(0,1)B_{n,k}(1)\right]\left|2\mathrm{i}\xi\right|<1.$ In particular,  equation (\ref{eqfdext}) corresponds to the characteristic function of the sum of a double-indexed weighted   random series of independent  chi--squared probability distributions,  with increasing sequence of degrees of freedom $\left\{ \Gamma(n,d),\ n\in \mathbb{N}_{0}\right\},$ and
  respective  weights defined from the elements of  the double-indexed  sequence $\left\{B_{n}(0,1)B_{n,k}(1),\ k\geq 1,\ n\in \mathbb{N}_{0}\right\}.$

In the next section, we derive a NCLT for the spatiotemporal functional $S_{T},$ given by
\begin{equation}
S_{T}=\frac{1}{d_{T}}\int_{\mathbb{T}(T)}\int_{\mathbb{S}_{d}(T^{\gamma })}(Z^{2}(\mathbf{x},t)-1)d\nu(\mathbf{x})dt,\quad T\geq 1,
\label{eq2}
\end{equation}
\noindent where $d_{T}=T^{\gamma (d-\alpha_{S})}T^{1-\alpha_{T}}\mathcal{L}_{1}(T)\mathcal{L}_{2}(T^{\gamma }),$  $\gamma \geq 0,$  with  $T^{\gamma }$ denoting, as before,  the  scaling factor of the zero--center homothetic transformation of $\mathbb{S}_{d}.$
 Thus, the variance scaling  $d_{T}$ will depend on the topological dimension $d$ of $\mathbb{S}_{d},$ the parameter $\gamma $ defining the spatial homothetic scaling factor, and on the parameters $\alpha_{T}$ and $\alpha_{S}$ characterizing the asymptotic power law of the unrestricted  spatiotemporal covariance tails. It also depends  on the positive slowly varying functions $\mathcal{L}_{i},$ $i=1,2$  (see \textbf{Assumption A1} below).

\section{NCLT via characteristic function}
\label{ncltchf}
This section provides the assumptions on the  unrestricted spatiotemporal covariance kernel  $C_{Z},$ to define  the spherical STRF setting where our NCLT  in manifolds is derived.
As commented, this result is based on the previous pure point spectral identities obtained, that will be applied in the  computation of the characteristics functions of the elements of $\left\{S_{T},\ T\geq 1\right\},$ as well as of their  limit.

\medskip

\noindent \textbf{Assumption A1}.  As $T\to \infty,$
\begin{eqnarray}&&C_{Z}(T^{\gamma }\mathbf{x},T\tau)\simeq \frac{\mathcal{L}_{1}(T)\mathcal{L}_{2}(T^{\gamma })}{|T\tau|^{\alpha_{T} }\|T^{\gamma }\mathbf{x}\|^{\alpha_{S} }}\nonumber\\ &&=\mathcal{O}\left(\mathcal{L}_{1}(T)\mathcal{L}_{2}(T^{\gamma })T^{-\gamma \alpha_{S}}T^{-\alpha_{T}}\right),\ 0<\alpha_{T}<1/2,\ 0<\alpha_{S}<d/2, \ T\to\infty,\label{cstas}\end{eqnarray}
\noindent where $f_{T}(\cdot) \simeq g_{T}(\cdot),$ $T\to \infty,$  means that the two function sequences converge pointwise to the same   limit, and
$\mathcal{L}_{i}$ $i=1,2,$ are positive slowly varying functions at
infinity such that
\begin{equation}
\lim_{T\rightarrow \infty }\mathcal{L}_{i}(Tu)/\mathcal{L}_{i}(T)=1,\quad u\in \mathbb{R}_{+},\ i=1,2.\label{eq2b}
\end{equation}

In Example 1 in Section \ref{examples}, one can consider the following identities:

\begin{eqnarray}
\varphi (u)&=&\frac{\widetilde{\mathcal{L}}_{2}(u)}{u^{\widetilde{\gamma} \nu}},\quad \widetilde{\mathcal{L}}_{2}(u)=\frac{u^{\widetilde{\gamma } \nu}}{(1+cu^{\widetilde{\gamma }})^{\nu }}\nonumber\\
\psi(u) &=& \left[\frac{\widetilde{\mathcal{L}}_{1}(u)}{u^{\alpha \beta}}\right]^{-1},\quad \widetilde{\mathcal{L}}_{1}(u)=\frac{u^{\alpha \beta}}{(1+au^{\alpha })^{\beta }}.
\label{Taubth}
\end{eqnarray}
\noindent Here,   $\widetilde{\mathcal{L}}_{i},$ $i=1,2,$ are positive continuous slowly varying functions at infinity, satisfying
equation (\ref{eq2b}). Then,  for $d=3,$ in this example,
\begin{eqnarray}&&C_{Z}(T^{\gamma }\mathbf{x},T\tau)= \mathcal{O}\left(\frac{\left[\widetilde{\mathcal{L}}_{1}(T^{2})\right]^{5/2}}{|T\tau|^{5\alpha \beta }}\frac{\widetilde{\mathcal{L}}_{2}(T^{2\gamma })}{\|T^{\gamma }\mathbf{x}\|^{2\widetilde{\gamma } \nu}}\right),\quad T\to \infty,\label{Taubth2}
\end{eqnarray}
\noindent obtaining  that \textbf{Assumption A1} holds for $\alpha_{T}=5\alpha \beta,$ with $0<5\alpha \beta <1/2,$ and $\alpha_{S}=2\widetilde{\gamma } \nu,$ such that  $0<2\widetilde{\gamma }\nu <3/2.$ In particular,   in this example,  $\mathcal{L}_{1}$ and $\mathcal{L}_{2}$ in \textbf{Assumption A1} are respectively  given by $\mathcal{L}_{1}(u)=\left[\widetilde{\mathcal{L}}_{1}(u^{2})\right]^{\frac{3}{2}+1},$
  and $\mathcal{L}_{2}(u)=\widetilde{\mathcal{L}}_{2}(u^{2}).$

  From equation (\ref{eqimlf}), one can similarly prove that the spatiotemporal covariance subfamily introduced in Example 2  satisfies
equation (\ref{cstas}) in \textbf{Assumption A1}, since its asymptotic behavior corresponds to consider  $\nu=1$ in Example 1.

\begin{remark}
Applying a Tauberian Theorem (see \cite{Doukhan},
and Theorems 4 and 11 in \cite{LeonenkoOlenko14}), the Fourier transforms of functions $\varphi $ and $\psi,$ in Examples 1 and 2, satisfy
 \begin{eqnarray}\widehat{\varphi}(\lambda )=\int_{\mathbb{R}^{d}}\exp\left(-i\left\langle \lambda , z\right\rangle \right) \varphi(\|z\|^{2})dz &\sim & c(1,2\widetilde{\gamma } \nu) \frac{\widetilde{\mathcal{L}}_{2}\left( \frac{1}{\|\lambda \|}\right)}{\|\lambda\|^{d-2\widetilde{\gamma } \nu}},\quad \|\lambda \|\to 0\label{asympsd0}\\
 \widehat{\psi}(\omega )=\int_{\mathbb{R}}\exp\left(-i\left\langle \omega , t\right\rangle \right) \psi(|t|^{2})dt &\sim &  c(1,2\alpha \beta)\frac{\widetilde{\mathcal{L}}_{1}\left( \frac{1}{|\omega|}\right)}{|\omega|^{1-2\alpha \beta}},\quad |\omega | \to 0,
  \label{asympsd}
\end{eqnarray}
\noindent where  $c\left( d,\theta \right) =\frac{\Gamma \left(\frac{d-\theta}{2}\right)}{\pi^{d/2}2^{\theta }\Gamma
(\theta/2)},$ with $\theta =2\widetilde{\gamma } \nu$ in (\ref{asympsd0})   for $0< 2\widetilde{\gamma } \nu <d,$ and $\theta =2\alpha \beta$  in (\ref{asympsd}) for $0<2\alpha \beta <1.$
\end{remark}

Under \textbf{Assumption A1}, the following asymptotic   trace formula holds as $T\to \infty,$ for each $m\geq 1,$
\begin{eqnarray}
&&
\frac{1}{d_{T}^{m}}\sum_{n=0}^{\infty }\sum_{k\geq 1}\Gamma(n,d)[B_{n}(0,T^{\gamma })B_{n,k}(T)]^{m}\nonumber\\
&&\simeq
\int_{\mathbb{S}_{d}^{m}(1)\times \mathbb{T}^{m}(1)}\prod_{i=1}^{m-1}\frac{\mathcal{L}_{1}(T|t_{i}-t_{i+1}|)\mathcal{L}_{2}(T^{\gamma }\|\mathbf{x}_{i}-\mathbf{x}_{i+1}\|)}{\mathcal{L}_{1}(T)\mathcal{L}_{2}(T^{\gamma })|t_{i}-t_{i+1}|^{\alpha_{T} }\|\mathbf{x}_{i}-\mathbf{x}_{i+1}\|^{\alpha_{S} }}\nonumber\\
&& \hspace*{2cm} \times
\frac{\mathcal{L}_{1}(T|t_{m}-t_{1}|)\mathcal{L}_{2}(T^{\gamma }\|\mathbf{x}_{m}-\mathbf{x}_{1}\|)}{\mathcal{L}_{1}(T)\mathcal{L}_{2}(T^{\gamma })|t_{m}-t_{1}|^{\alpha_{T} }\|\mathbf{x}_{m}-\mathbf{x}_{1}\|^{\alpha_{S} }}
\prod_{i=1}^{m}d\nu(\mathbf{x}_{i})dt_{i}.
\label{tracef}
\end{eqnarray}

\medskip

The next assumption is of technical nature, and will be considered  in the proof of Theorem \ref{pr1} below, allowing the application of  Dominated Convergence Theorem arguments. It does not suppose a big restriction on our covariance function setting, since it holds for a wide family of  slowly varying functions $\mathcal{L}.$

\medskip

\noindent \textbf{Assumption A2}. Assume that the slowly varying functions $\mathcal{L}_{i},$ $i=1,2,$  in \textbf{Assumption A1}, satisfy
\begin{eqnarray}&&\hspace*{-1cm}\int_{\mathbb{S}_{d}^{m}(1)\times \mathbb{T}^{m}(1)}\prod_{i=1}^{m-1}\frac{\mathcal{L}_{1}(T|t_{i}-t_{i+1}|)\mathcal{L}_{2}(T^{\gamma }\|\mathbf{x}_{i}-\mathbf{x}_{i+1}\|)}{\mathcal{L}_{1}(T)\mathcal{L}_{2}(T^{\gamma })|t_{i}-t_{i+1}|^{\alpha_{T} }\|\mathbf{x}_{i}-\mathbf{x}_{i+1}\|^{\alpha_{S} }}\nonumber\\
&& \hspace*{3cm}\times
\frac{\mathcal{L}_{1}(T|t_{m}-t_{1}|)\mathcal{L}_{2}(T^{\gamma }\|\mathbf{x}_{m}-\mathbf{x}_{1}\|)}{\mathcal{L}_{1}(T)\mathcal{L}_{2}(T^{\gamma })|t_{m}-t_{1}|^{\alpha_{T} }\|\mathbf{x}_{m}-\mathbf{x}_{1}\|^{\alpha_{S} }
\prod_{i=1}^{m}d\nu(\mathbf{x}_{i})dt_{i}}
\nonumber\\ &&\hspace*{-1cm}\leq \mathcal{M}\int_{\mathbb{S}_{d}^{m}(1)\times \mathbb{T}^{m}(1)}\prod_{i=1}^{m-1}\frac{ 1}{|t_{i}-t_{i+1}|^{\alpha_{T} }\|\mathbf{x}_{i}-\mathbf{x}_{i+1}\|^{\alpha_{S} }}\frac{1}{|t_{m}-t_{1}|^{\alpha_{T} }\|\mathbf{x}_{m}-\mathbf{x}_{1}\|^{\alpha_{S} }}
\prod_{i=1}^{m}d\nu(\mathbf{x}_{i})dt_{i},\nonumber\end{eqnarray}
\noindent for some positive constant $\mathcal{M}.$

\begin{remark}
From equation (\ref{Taubth}), $\mathcal{L}_{1}(u)=\left[\widetilde{\mathcal{L}}_{1}(u^{2})\right]^{\frac{3}{2}+1}$
  and $\mathcal{L}_{2}(u)=\widetilde{\mathcal{L}}_{2}(u^{2}),$ in Examples 1 and 2, satisfy \textbf{Assumption A2}.
\end{remark}

\subsection{NCLT in the pure point spectral domain}
Let  $A_{T}$ be the  linear functional of the Gaussian subordinated STRF $\mathcal{J}(Z(\mathbf{x},t)),$ with $\mathcal{J}$ having Hermite rank
$m$ defined by:

 \begin{eqnarray}
 A_{T}&=&\frac{1}{d_{T}(m)}\int_{\mathbb{T}(T)}\int_{\mathbb{S}_{d}(T^{\gamma })}\mathcal{J}(Z(\mathbf{x},t))d\nu(\mathbf{x})dt,
 \label{HESTF}
 \end{eqnarray}
 \noindent where  $d_{T}(m)=T^{\gamma (d-(m\alpha_{S}/2))}T^{1-(m\alpha_{T}/2)}\mathcal{L}^{m/2}_{1}(T)\mathcal{L}^{m/2}_{2}(T^{\gamma }),$
   $E[\mathcal{J}^{2}(Z(\mathbf{x},t))]<\infty,$ $\mathbf{x}\in \mathbb{R}^{d+1},$ and $t\in \mathbb{R}.$
 Thus, $\mathcal{J}$ admits the following  series expansion  in  $\mathcal{L}_{2}(\Omega,\mathcal{A},P),$ with respect to  the orthogonal Hermite basis $\left\{ H_{q},\ q\geq 0\right\}:$
 \begin{equation}\mathcal{J}(z)=\mathcal{J}_{0}+\sum_{q=m}^{\infty }\frac{\mathcal{J}_{q}}{q!}H_{q}(z),\ z\in \mathbb{R},\  \mathcal{J}_{0}=E[\mathcal{J}],\ \mathcal{J}_{q}=\int_{\mathbb{R}}H_{q}(z)\mathcal{J}(z)\phi(z)dz,\quad  q\geq m,\label{Hefind}\end{equation}
\noindent where, for $q\geq 0,$  $H_{q}$  is the   the Hermite polynomial   of order $q$ satisfying
\begin{equation}
\frac{d^{n}\phi}{dz^{n}}(z)=(-1)^{n}H_{n}(z)\phi(z).
\label{hq}
\end{equation}
\noindent Hence,
\begin{eqnarray}
H_{0}(z)&=& 1,\ H_{1}(z)=z,\ H_{2}(z)=z^{2}-1\nonumber\\
H_{3}(z)&=& z^{3}-3z,\ H_{4}(z)=z^{4}-6z^{2}+3,\dots ,\label{exherpol}
\end{eqnarray}
\noindent and
\begin{eqnarray}
 A_{T}&=& \frac{1}{d_{T}(m)}\left[2T^{1+d\gamma }|\mathbb{S}_{d}(1)|\mathcal{J}_{0}+\sum_{n=m}^{\infty}\frac{\mathcal{J}_{n}}{n!}
\int_{\mathbb{T}(T)}\int_{\mathbb{S}_{d}(T^{\gamma })}H_{n}(Z(\mathbf{x},t))d\nu(\mathbf{x})dt\right].\nonumber\end{eqnarray}

   In the next proposition, Theorem 1 in \cite{LeonenkoRuizMedina25} is reformulated in our setting.
\begin{proposition}
\label{th3bb} Under \textbf{Assumptions A0-A1},  considering $\mathcal{J}(z)$  has Hermite rank $m,$  the random
variables
\begin{eqnarray}
Y_{T} &=&\frac{\widetilde{A}_{T}-\mathbb{E}[\widetilde{A}_{T}] }{|\mathcal{J}_{m}|\sigma _{m,K}(T)/m!}  \label{funct1} \\
&&  \notag \\
&&\hspace*{-3.75cm}\mbox{and}  \notag \\
&&  \notag \\
&&Y_{m,T}=\frac{\mbox{sgn}(\mathcal{J}_{m})\int_{\mathbb{T}(T)}\int_{\mathbb{S}_{d}(T^{\gamma })}H_{m}(Z(x,t))dxdt}{\sigma _{m,K}(T)}  \label{funct2}
\end{eqnarray}
\noindent have the same limiting distributions as $T\rightarrow \infty $ (if
one of it exists). Here, \begin{eqnarray}\sigma _{m,K}(T)&=&\mbox{Var}\left(\int_{\mathbb{T}(T)}\int_{\mathbb{S}_{d}(T^{\gamma })}H_{m}(Z(\mathbf{x},t))d\nu(\mathbf{x})dt\right)\nonumber\\
\widetilde{A}_{T} &=&d_{T}(m)A_{T},\quad d_{T}(m)=T^{\gamma (d-(m\alpha_{S}/2))}T^{1-(m\alpha_{T}/2)}\mathcal{L}^{m/2}_{1}(T)\mathcal{L}^{m/2}_{2}(T^{\gamma }), \nonumber\end{eqnarray}
\noindent where $A_{T}$ has been introduced in equation (\ref{HESTF}), and $\sigma _{m,K}(T)=d_{T}(m)(1+o(1)),$ as $T\to \infty.$
\end{proposition}

The proof is straightforward since, under  \textbf{Assumptions A0-A1}, \textbf{Conditions 1,2,3 and 4} in  Theorem 1 in \cite{LeonenkoRuizMedina25} hold.

\medskip

Applying Proposition \ref{th3bb}, the next result characterizes the asymptotic behavior of $A_{T}$ for Hermite rank $m=2.$ Thus,
    $S_{T}$ in (\ref{eq2}), and  $A_{T}$ in (\ref{HESTF})   display the same asymptotic behavior, under $m=2$.

\begin{theorem}
\label{pr1} Under \textbf{Assumptions A0--A2},  the following weak convergence holds:
$$S_{T}\to_{D}S_{\infty},\quad T\to \infty,$$
\noindent where
$$S_{T}=\frac{1}{d_{T}}\int_{\mathbb{T}(T)}\int_{\mathbb{S}_{d}(T^{\gamma })}(Z^{2}(\mathbf{x},t)-1)d\nu(\mathbf{x})dt,\quad T\geq 1,$$
\noindent with $d_{T}=T^{\gamma (d-\alpha_{S})}T^{1-\alpha_{T}}\mathcal{L}_{1}(T)\mathcal{L}_{2}(T^{\gamma }),$ and the limit random variable $S_{\infty}$ has characteristic function given by
\begin{equation}
\psi (\xi )=E\left[\exp\left(i\xi  S_{\infty}\right)\right]=\exp \left( \frac{1}{2}\sum_{m=2}^{\infty }\frac{\left( 2\mathrm{i}%
\xi \right) ^{m}}{m}c_{m}\right),\quad \xi \in \mathbb{R}. \label{chf}
\end{equation}%
\noindent Here, for $m\geq 2,$
\begin{eqnarray}
c_{m} &=&\int_{\mathbb{S}_{d}^{m}(1)\times \mathbb{T}^{m}(1)}\prod_{i=1}^{m-1}\frac{ 1}{|t_{i}-t_{i+1}|^{\alpha_{T} }\|\mathbf{x}_{i}-\mathbf{x}_{i+1}\|^{\alpha_{S} }}\frac{1}{|t_{m}-t_{1}|^{\alpha_{T} }\|\mathbf{x}_{m}-\mathbf{x}_{1}\|^{\alpha_{S} }}
\nonumber\\
&&\hspace*{8cm}\times
\prod_{i=1}^{m}d\nu(\mathbf{x}_{i})dt_{i}.
\label{eqcoefchfbcc}
\end{eqnarray}
\end{theorem}
\begin{proof}
Note that, applying   asymptotic spectral properties of Riesz kernel  restricted to a compact set (see, e.g.,  \cite{Dostanić98}  and \cite{Widom63}),
for $0<\alpha_{T}<1/2,$ and  $0<\alpha_{S}<d/2,$ the integral operator
  $\mathcal{K}^{2}_{\alpha_{S},\alpha_{T}}=\mathcal{K}_{\alpha_{S},\alpha_{T}}\star\mathcal{K}_{\alpha_{S},\alpha_{T}}$ on $L^{2}(\mathbb{S}_{d}(1)\times \mathbb{T}(1), d\nu\otimes dt),$ with kernel
$k^{\star(2)}_{\alpha_{S},\alpha_{T}},$ given by
\begin{equation}k^{\star (2)}_{\alpha_{S},\alpha_{T}}(\mathbf{x},\mathbf{y},t,s)=\frac{1}{\|\mathbf{x}-\mathbf{y}\|^{2\alpha_{S}}|t-s|^{2\alpha_{T}}},\quad \mathbf{x},\mathbf{y}\in \mathbb{S}_{d}(1),\ t,s\in \mathbb{T}(1),\label{keyop}
\end{equation}
\noindent has finite  trace norm $\left\|\mathcal{K}^{2}_{\alpha_{S},\alpha_{T}}\right\|_{L^{1}(L^{2}(\mathbb{S}_{d}(1)\times \mathbb{T}(1), d\nu \otimes dt))}.$ Here,
 $\star $ denotes convolution, and
 \begin{eqnarray}&&
 \hspace*{-0.5cm}\left\|\mathcal{K}^{2}_{\alpha_{S},\alpha_{T}}\right\|_{L^{1}(L^{2}(\mathbb{S}_{d}(1)\times \mathbb{T}(1), d\nu \otimes dt))}=
 \int_{[\mathbb{S}_{d}(1)\times \mathbb{T}(1)]^{2}}\frac{1}{\|\mathbf{x}-\mathbf{y}\|^{2\alpha_{S}}|t-s|^{2\alpha_{T}}}dsdtd\nu(\mathbf{x})d\nu(\mathbf{y}).\nonumber\\
 \label{trace2form}
 \end{eqnarray}

 \noindent
The proof of this result then  follows from  the application of Dominated Convergence and Continuous Mapping Theorems. Specifically,
let us first compute, for each $\xi\in \mathbb{R},$  the characteristic function $\psi_{S_{T}} (\xi )$ of $S_{T},$  $T\geq 1.$
From equations (\ref{KLexp}), (\ref{Ysquare}) and (\ref{tvarcf}),

\begin{eqnarray}
&&\psi_{S_{T}} (\xi )=\left[\exp\left(i\xi S_{T}\right)\right]=\exp\left(-i\xi \frac{\sum_{n=0}^{\infty }\Gamma(n,d)B_{n}(0,T^{\gamma })\sum_{k\geq 1}B_{n,k}(T)}{d_{T}}\right)\nonumber\\ &&\hspace*{0.5cm}\times \exp \left( \frac{1}{2}\sum_{m=1}^{\infty }\frac{((2\mathrm{i}\xi)/d_{T} )^{m}}{m}
\sum_{n=0}^{\infty }\Gamma(n,d)[B_{n}(0,T^{\gamma })]^{m}\sum_{k\geq 1}[B_{n,k}(T)]^{m}\right)\nonumber\\
&&=\exp\left(-i\xi c_{1}(T)\right)\exp \left( \frac{1}{2}\sum_{m=1}^{\infty }\frac{(2\mathrm{i}\xi)^{m}}{m}c_{m}(T)\right)\nonumber\\
&&=\exp \left( \frac{1}{2}\sum_{m=2}^{\infty }\frac{(2\mathrm{i}\xi)^{m}}{m}c_{m}(T)\right),\label{chtst}\end{eqnarray}
\noindent with $$c_{m}(T)=\frac{1}{d^{m}_{T}}\sum_{n=0}^{\infty }\Gamma(n,d)[B_{n}(0,T^{\gamma })]^{m}\sum_{k\geq 1}[B_{n,k}(T)]^{m},\ T\geq 1,\ m\geq 1.
$$

Under \textbf{Assumptions A0--A2}, from equation (\ref{tracef}), for every $m\geq 2,$ considering  $T$ sufficiently large,

\begin{eqnarray}
c_{m}(T)&\leq &\mathcal{M} \int_{\mathbb{S}_{d}^{m}(1)\times \mathbb{T}^{m}(1)}\prod_{i=1}^{m-1}\frac{ 1}{|t_{i}-t_{i+1}|^{\alpha_{T} }\|\mathbf{x}_{i}-\mathbf{x}_{i+1}\|^{\alpha_{S} }}\frac{1}{|t_{m}-t_{1}|^{\alpha_{T} }\|\mathbf{x}_{m}-\mathbf{x}_{1}\|^{\alpha_{S} }}
\nonumber\\
&&\hspace*{9.5cm}\times \prod_{i=1}^{m}d\nu(\mathbf{x}_{i})dt_{i}\nonumber\\
&=&\mathcal{M}\left\|\mathcal{K}^{m}_{\alpha_{S},\alpha_{T}}\right\|_{L^{1}(L^{2}(\mathbb{S}_{d}(1)\times \mathbb{T}(1), d\nu\otimes dt))}.
\label{eqcmdcth}\end{eqnarray}

As pointed out before, $\left\|\mathcal{K}^{2}_{\alpha_{S},\alpha_{T}}\right\|_{L^{1}(L^{2}(\mathbb{S}_{d}(1)\times \mathbb{T}(1), d\nu \otimes dt))}<\infty,$ which means that the logarithm of the Fredholm determinant of $\mathcal{K}^{2}_{\alpha_{S},\alpha_{T}}$  is finite on a neighborhood of  zero. That is, for $\xi$ such that $|\xi |<\frac{1}{2\left\|\mathcal{K}^{2}_{\alpha_{S},\alpha_{T}}\right\|_{L^{1}(L^{2}(\mathbb{S}_{d}(1)\times \mathbb{T}(1), d\nu \otimes dt))}},$
we have
\begin{equation}
\sum_{m=2}^{\infty }\frac{(2\mathrm{i}\xi)^{m}}{m}c_{m}=\sum_{m=2}^{\infty }\frac{(2\mathrm{i}\xi)^{m}}{m}\left\|\mathcal{K}^{m}_{\alpha_{S},\alpha_{T}}\right\|_{L^{1}(L^{2}(\mathbb{S}_{d}(1)\times \mathbb{T}(1), d\nu \otimes dt))}
<\infty.\label{lfd}
\end{equation}

\noindent Therefore, from equations (\ref{eqcmdcth})-(\ref{lfd}), one can apply Dominated Convergence Theorem  leading  to  \begin{eqnarray}&&\lim_{T\to \infty}\sum_{m=2}^{\infty }\frac{(2\mathrm{i}\xi)^{m}}{m}c_{m}(T)=
\sum_{m=2}^{\infty }\frac{(2\mathrm{i}\xi)^{m}}{m}\lim_{T\to \infty}c_{m}(T)=\sum_{m=2}^{\infty }\frac{(2\mathrm{i}\xi)^{m}}{m}c_{m},\label{limidcth}
\end{eqnarray}
\noindent for every $\xi $ such that   $|\xi |<\frac{1}{2\left\|\mathcal{K}^{2}_{\alpha_{S},\alpha_{T}}\right\|_{L^{1}(L^{2}(\mathbb{S}_{d}(1)\times \mathbb{T}(1), d\nu \otimes dt))}}.$
Finally,  for $\xi $ such that    $|\xi |<\frac{1}{2\left\|\mathcal{K}^{2}_{\alpha_{S},\alpha_{T}}\right\|_{L^{1}(L^{2}(\mathbb{S}_{d}(1)\times \mathbb{T}(1), d\nu \otimes dt))}},$ applying Continuous Mapping Theorem,   we obtain  $$\lim_{T\to \infty}\psi_{S_{T}} (\xi )=\psi (\xi ),$$ \noindent as we wanted to prove.
\end{proof}

\section{Mean--square convergence   for STRFs on compact sets  of positive Lebesgue measure}
\label{ncltspectr}

In this section we work in  the second--order continuous spectral domain  of the unrestricted STRF $Z.$  Conditions are then  formulated in terms of  the local  behavior  at zero frequency of its spectral density $f_{Z}.$  Indeed, the results derived here hold  for any nontrivial compact convex regular domain $\mathcal{K}$ of $\mathbb{R}^{d+1},$ containing the null vector $\mathbf{0}\in \mathbb{R}^{d+1}.$
By this reason, in the subsequent development we will adopt the notation $\mathcal{K}$ to refer to the spatial domain of $\mathbb{R}^{d+1}$ where our STRF family is restricted.
The case of $\mathbb{S}_{d}(1)$ deserves a separate analysis, and  will be  addressed   in Section \ref{discussion}.

The following  assumptions are required regarding the asymptotic  local behavior of $f_{Z}$ at spatiotemporal  zero frequency.

\begin{itemize}
\item[\textbf{B1}] The spectral density  $f_{Z}$   is  such that, for  each $\boldsymbol{\lambda }\in \mathbb{R}^{d+1},$ and  $\omega \in \mathbb{R},$
\begin{eqnarray}&&\lim_{T\to \infty}\frac{f_{Z}(\|T^{-\gamma }\boldsymbol{\lambda }\|,T^{-1}\omega )}{T^{\gamma (d+1-\alpha_{S})+(1-\alpha_{T})}\mathcal{L}_{S}(T^{\gamma })\mathcal{L}_{T}(T)}=
\|\boldsymbol{\lambda}\|^{-(d+1-\alpha_{S})}|\omega|^{-(1-\alpha_{T})},
\nonumber\end{eqnarray}
 \noindent where  $\gamma \geq 0,$ $\alpha_{S}\in (0, (d+1)/2),$ and $\alpha _{T}\in (0,1/2).$

 \end{itemize}

\begin{itemize}

\item[\textbf{B2}] The function $f_{Z}$ is a decreasing function, with respect to the Euclidean distance to zero frequency at a neighborhood of  this frequency.
 \end{itemize}
 \begin{remark}
Note that the decreasing monotone behavior  of  $f_{Z}$ in a neighbourhood of zero frequency in \textbf{Assumption B2},  with $f_{Z}$ being continuous except for $(\omega , \boldsymbol{\lambda })=(0,\mathbf{0}),$ implies that, if the covariance function of $Z$ satisfies   an \textbf{Assumption A1}-type condition, then,   by Tauberian Theorem 4 in \cite{LeonenkoOlenko13}, \textbf{B1}  holds (see also \cite{Olenko05}; \cite{LeonenkoOlenko14}).
 \end{remark}

The next result establishes the convergence in the mean--square sense  of  the functional
 $$S_{T}=\frac{1}{\mathcal{L}_{S}(T^{\gamma })T^{\gamma (d+1-\alpha_{S})}\mathcal{L}_{T}(T)T^{1-\alpha_{T}}}\int_{\mathbb{T}(T)\times T^{\gamma }\mathcal{K}} \left(Z^{2}(\mathbf{x},t)-1\right)d\nu_{\mathcal{K}}(\mathbf{x})dt,$$
\noindent to a random variable $S_{\infty}$ admitting a  double Wiener-It{\^o} stochastic integral representation. Here,
$d\nu_{\mathcal{K}}$ denotes the Lebesgue measure in $\mathbb{R}^{d+1}.$
 \begin{theorem}
 \label{th3}
 Under \textbf{B1}--\textbf{B2}, as $T\to \infty,$  $$E[S_{T}-S_{\infty}]^{2}\to 0,$$   \noindent where
 \begin{eqnarray}&&S_{\infty}=2|\mathcal{K}|c(\alpha_{S},\alpha_{T},d+1)\int_{(\mathbb{R}^{d+1}\times \mathbb{R})\times (\mathbb{R}^{d+1}\times \mathbb{R})}^{\prime \prime }\widehat{1}_{\mathbb{T}(1)}(\omega_{1}+\omega_{2})\widehat{1}_{\mathcal{K}}(\boldsymbol{\lambda}_{1}+\boldsymbol{\lambda}_{2})\nonumber\\ &&\hspace*{5cm}\times \prod_{j=1}^{2}\frac{1}{\|\boldsymbol{\lambda_{j} }\|^{(d+1-\alpha_{S})/2}}\frac{1}{|\omega_{j} |^{(1-\alpha_{T})/2}}
W(d\boldsymbol{\lambda}_{j},d\omega_{j}),\nonumber\end{eqnarray}
 \noindent where     $|\mathcal{K}|=\int_{\mathcal{K}}d\nu_{\mathcal{K}}(\mathbf{x}),$ and with    $\int_{(\mathbb{R}^{d+1}\times \mathbb{R})\times (\mathbb{R}^{d+1}\times \mathbb{R})}^{\prime \prime }$  meaning that one does not
integrate
on the hyperdiagonals. Here, $W$ denotes a complex spatiotemporal Wiener measure on $\mathbb{R}^{d+1}\times \mathbb{R},$  i.e.,  a complex white noise Gaussian measure with Lebesgue control measure on $\mathbb{R}^{d+1}\times \mathbb{R}.$  Functions $\widehat{1}_{\mathcal{K}}$ and $\widehat{1}_{\mathbb{T}(1)}$ respectively denote the  characteristic functions  of the uniform probability distributions of the sets $\mathcal{K}$ and $\mathbb{T}(1).$ The constant $$c(\alpha_{S},\alpha_{T},d+1)=\frac{\Gamma \left(\frac{d+1-\alpha_{S}}{2}\right)}{\pi^{(d+1)/2}2^{\alpha_{S}}\Gamma (\alpha_{S}/2)}  \frac{\Gamma \left(\frac{1-\alpha_{T}}{2}\right)}{\pi^{1/2}2^{\alpha_{T}}\Gamma (\alpha_{T}/2)}.$$
 \end{theorem}

 \begin{proof}
  Under \textbf{B1}--\textbf{B2}, for $T$ sufficiently large, $$Q_{T}(\boldsymbol{\lambda }_{1}, \boldsymbol{\lambda }_{2},\omega_{1},\omega_{2})=\left[\prod_{j=1}^{2}d_{T}^{-1}\frac{f_{Z}^{1/2}(\|T^{-\gamma }\boldsymbol{\lambda }_{j}\|,T^{-1}\omega_{j} )}{\left[\|\boldsymbol{\lambda_{j}}\|^{(d+1-\alpha_{S})/2}|\omega_{j} |^{(1-\alpha_{T})/2}\right]^{-1}} -1\right]^{2}\leq 1,$$
  \noindent where $d_{T}=\mathcal{L}_{S}(T^{\gamma })T^{\gamma (d+1-\alpha_{S})}\mathcal{L}_{T}(T)T^{1-\alpha_{T}}.$
 Hence, from the asymptotic spectral properties of Riesz kernel  restricted to a compact set (see \cite{Dostanić98}  and \cite{Widom63}),
  \begin{eqnarray}
  &&\hspace*{-1cm}E\left[S_{T}-S_{\infty}\right]^{2}=\int_{\mathbb{R}^{2(d+1)}\times \mathbb{R}^{2}}
  \left|\widehat{1}_{\mathbb{T}(1)}(\omega_{1}+\omega_{2})\widehat{1}_{\mathcal{K}}(\boldsymbol{\lambda}_{1}+\boldsymbol{\lambda}_{2})
  \right|^{2}
  \prod_{j=1}^{2}\|\boldsymbol{\lambda}_{j} \|^{- (d+1-\alpha_{S})}
  \nonumber\\
  &&
  \hspace*{1.75cm} \times [2|\mathcal{K}|c(d+1,\alpha_{S},\alpha_{T})]^{2} |\omega_{j} |^{-(1-\alpha_{T})}Q_{T}(\boldsymbol{\lambda }_{1}, \boldsymbol{\lambda }_{2},\omega_{1},\omega_{2})\prod_{j=1}^{2}d\boldsymbol{\lambda }_{j}d\omega_{j}\nonumber\\
  &&\leq  [2|\mathcal{K}|c(d+1,\alpha_{S},\alpha_{T})]^{2}\int_{\mathbb{R}^{2(d+1)}\times \mathbb{R}^{2}}
   \left|\widehat{1}_{\mathbb{T}(1)}(\omega_{1}+\omega_{2})\widehat{1}_{\mathcal{K}}(\boldsymbol{\lambda}_{1}+\boldsymbol{\lambda}_{2})
  \right|^{2}\nonumber\\
  &&\hspace*{3cm}\times \prod_{j=1}^{2}\|\boldsymbol{\lambda}_{j}\|^{-(d+1-\alpha_{S})}|\omega_{j} |^{-(1-\alpha_{T})}\prod_{j=1}^{2}d\boldsymbol{\lambda }_{j}d\omega_{j}
  \nonumber\\
  &&= \int_{\mathbb{T}(1)\times \mathbb{T}(1)}\int_{\mathcal{K}\times \mathcal{K} }\frac{1}{|t-s|^{2\alpha_{T}}\|\mathbf{x}-\mathbf{y}\|^{2\alpha_{S}}}d\nu_{\mathcal{K}}(\mathbf{x})d\nu_{\mathcal{K}}(\mathbf{y})dtds <\infty,
  \end{eqnarray}

  \noindent since  $\alpha_{S}\in (0, (d+1)/2),$ and $\alpha _{T}\in (0,1/2)$  under \textbf{B1}. Thus,  the corresponding trace norm of this operator
  is finite  on $L^{2}(\mathcal{K}\times \mathbb{T}(1))$.
  Dominated Convergence Theorem then   leads to the desired result as  $T\to \infty.$
 \end{proof}
\section{Mean--square convergence  for STRFs on $\mathbb{S}_{d}(1)$ }
\label{mssphere}
Through this section  we will work under \textbf{Assumptions A0--A2}. In particular, equations (\ref{KLexp}) and (\ref{Ysquare}) hold. The next result provides the convergence in the mean--square sense of $S_{T},$ introduced in equation  (\ref{eq2}), to $S_{\infty}$  in Theorem \ref{pr1}. The series expansion of $S_{\infty }$ is also provided.

\begin{theorem}
\label{thMS}
Under  \textbf{Assumptions A0--A2}, the functional $S_{T}$ in  (\ref{eq2})  also converges in the mean--square sense to  $S_{\infty}$ with characteristic function (\ref{chf}). Furthermore,
$S_{\infty}$ admits the following series representation
\begin{eqnarray}
S_{\infty}\underset{\mathcal{L}^{2}(\Omega,\mathcal{A},P)}{=}\sum_{n=0}^{\infty }
\widetilde{B}_{n}(0)\sum_{j=1}^{\Gamma (n,d)}\sum_{k\geq 1}\widetilde{B}_{k}\eta^{2}_{n,j,k}-\frac{4\pi^{\frac{d+1}{2}}}{\Gamma \left(%
\frac{d+1}{2}\right)},\label{mslimit}
\end{eqnarray}
\noindent where, as before, $\left\{ \eta_{n,j,k}, \  k\geq 1, \  j=1,\dots,\Gamma (n,d),\ n\geq 0\right\}$ defines a sequence of independent and identically distributed standard Gaussian random variables, and $\left\{ \widetilde{B}_{n}(0),\ n\geq 0\right\}$ and $\left\{\widetilde{B}_{k},\ k\geq 1\right\},$ respectively satisfy
\begin{eqnarray}
&& \int_{\mathbb{S}_{d}(1)}\frac{1}{\|\mathbf{x}-\mathbf{y}\|^{\alpha_{S}}}S_{n,j}^{d}(\mathbf{y})d\nu(\mathbf{y})=
\widetilde{B}_{n}(0)S_{n,j}^{d}(\mathbf{x}),\
\mathbf{x}\in \mathbb{S}_{d}(1)
\label{eig1}\\
&&\int_{\mathbb{T}(1)}\frac{1}{|t-s|^{\alpha_{T}}}\phi_{k}(s)ds=\widetilde{B}_{k}\phi_{k}(t),\ t\in \mathbb{T}(1).
\label{eig2}\end{eqnarray}
\noindent Hence,
\begin{eqnarray}&&\sum_{n\geq 0}\Gamma(n,d)[\widetilde{B}_{n}(0)]^{2}\sum_{k\geq 1}[\widetilde{B}_{k}]^{2}
=\int_{\mathbb{S}_{d}^{2}(1)\times \mathbb{T}^{2}(1)}k^{\star (2)}_{\alpha_{S},\alpha_{T}}(\mathbf{x},\mathbf{y},t,s)d\nu(\mathbf{x})d\nu(\mathbf{y})dt ds
\nonumber\\
&&
=\left\|\mathcal{K}^{2}_{\alpha_{S},\alpha_{T}}\right\|_{L^{1}(L^{2}(\mathbb{S}_{d}(1)\times \mathbb{T}(1),d\nu\otimes dt))}<\infty,\label{tr2}\end{eqnarray}
\noindent where kernel $k^{\star (2)}_{\alpha_{S},\alpha_{T}}$ of integral operator $\mathcal{K}^{2}_{\alpha_{S},\alpha_{T}}$ has been introduced in equation (\ref{keyop}).
\end{theorem}
\begin{remark}
\label{rr1}
The asymptotic behavior of $\left\{\widetilde{B}_{k},\ k\geq 1\right\}$ satisfying eigenequation (\ref{eig2}) was characterized in \cite{Dostanić98}. Furthermore, a similar asymptotic behavior, depending on the topological  dimension $d,$ is displayed  by $\left\{\widetilde{B}_{n}(0),\ n\geq 0\right\}$ in equation (\ref{eig1}) (see also \cite{Widom63}).
\end{remark}
\begin{proof}
From equations (\ref{KLexp}), (\ref{Ysquare}) and (\ref{mslimit}), under \textbf{Assumptions A0-A1},
\begin{eqnarray}
&&E\left[(S_{T}-S_{\infty})^{2}\right]= E\left[\left(\sum_{n=0}^{\infty}\sum_{j=1}^{\Gamma (n,j)}\sum_{k\geq 1}
\left[\frac{B_{n}(0,T^{\gamma })B_{n,k}(T)}{d_{T}}-\widetilde{B}_{n}(0)\widetilde{B}_{k}\right]\eta _{n,j,k}^{2}\right.\right.\nonumber\\
&&\hspace*{5cm}\left.\left.-\left[\frac{|\mathbb{S}_{d}(T^{\gamma })||\mathbb{T}(T)|}{d_{T}}-|\mathbb{S}_{d}(1)||\mathbb{T}(1)|\right]\right)^{2}\right]\nonumber\\
&&=\sum_{n=0}^{\infty}\sum_{j=1}^{\Gamma (n,j)}\sum_{k\geq 1}
\left[\frac{B_{n}(0,T^{\gamma })B_{n,k}(T)}{d_{T}}-\widetilde{B}_{n}(0)\widetilde{B}_{k}\right]^{2}E\left[\eta _{n,j,k}^{4}\right]
\nonumber\\
&&-2\left(\sum_{n=0}^{\infty}\sum_{j=1}^{\Gamma (n,j)}\sum_{k\geq 1}
\left[\frac{B_{n}(0,T^{\gamma })B_{n,k}(T)}{d_{T}}-\widetilde{B}_{n}(0)\widetilde{B}_{k}\right]E\left[\eta _{n,j,k}^{2}\right]\right)\nonumber\\&&\hspace*{0.1cm}\times \left[\frac{|\mathbb{S}_{d}(T^{\gamma })||\mathbb{T}(T)|}{d_{T}}-|\mathbb{S}_{d}(1)||\mathbb{T}(1)|\right]+\left[\frac{|\mathbb{S}_{d}(T^{\gamma })||\mathbb{T}(T)|}{d_{T}}-|\mathbb{S}_{d}(1)||\mathbb{T}(1)|\right]^{2}\nonumber\\
&&\leq 2\sum_{n=0}^{\infty}\Gamma (n,j)\sum_{k\geq 1}
\left[\frac{B_{n}(0,T^{\gamma })B_{n,k}(T)}{d_{T}}-\widetilde{B}_{n}(0)\widetilde{B}_{k}\right]^{2}\nonumber\\
&&+\left[  \sum_{n=0}^{\infty}\Gamma (n,j)\sum_{k\geq 1}
\left[\frac{B_{n}(0,T^{\gamma })B_{n,k}(T)}{d_{T}}-\widetilde{B}_{n}(0)\widetilde{B}_{k}\right]\right.\nonumber\\
&&\hspace*{4cm}\left.-\left[\frac{|\mathbb{S}_{d}(T^{\gamma })||\mathbb{T}(T)|}{d_{T}}-|\mathbb{S}_{d}(1)||\mathbb{T}(1)|\right]\right]^{2}\nonumber\\
&&=2\sum_{n=0}^{\infty}\Gamma (n,j)\sum_{k\geq 1}
\left[\frac{B_{n}(0,T^{\gamma })B_{n,k}(T)}{d_{T}}-\widetilde{B}_{n}(0)\widetilde{B}_{k}\right]^{2},
\label{fourthmtrc}
\end{eqnarray}
\noindent where we have applied $\|\mathcal{A}^{m}\|_{L^{1}(H)}\leq \|\mathcal{A}\|_{L^{1}(H)}^{m},$ for $m=2,$ with $\mathcal{A}$ being
a trace operator on the separable Hilbert space $H,$ and that the unrestricted STRF $Z$ has variance one (see also (\ref{tracef})).  Also, from (\ref{tr2}), where the asymptotic spectral properties of Riesz kernel restricted to a compact set  are applied (see Remark \ref{rr1}),
\begin{eqnarray}\sum_{n=0}^{\infty}\Gamma (n,j)\sum_{k\geq 1}\left[\widetilde{B}_{n}(0)\widetilde{B}_{k}\right]^{2}
&=& \left\|\mathcal{K}^{2}_{\alpha_{S},\alpha_{T}}\right\|_{L^{1}(L^{2}(\mathbb{S}_{d}(1)\times \mathbb{T}(1),d\nu\otimes dt))}<\infty,
\label{fourthmtrcb}
\end{eqnarray}
\noindent  \noindent which implies, under \textbf{Assumption A2}, that  \begin{eqnarray}&&\sum_{n=0}^{\infty}\Gamma (n,j)\sum_{k\geq 1}
\left[\frac{B_{n}(0,T^{\gamma })B_{n,k}(T)}{d_{T}}\right]^{2}<\infty.\label{fourthmtrcb2}
\end{eqnarray}

Thus, from equation (\ref{fourthmtrc}), applying   triangle inequality  in terms of  the Hilbert- Schmidt operator norm,  keeping in mind \textbf{Assumption A2},
\begin{eqnarray}&&
E\left[(S_{T}-S_{\infty})^{2}\right]\leq 2\sum_{n=0}^{\infty}\Gamma (n,j)\sum_{k\geq 1}
\left[\frac{B_{n}(0,T^{\gamma })B_{n,k}(T)}{d_{T}}-\widetilde{B}_{n}(0)\widetilde{B}_{k}\right]^{2}\nonumber\\
&&\leq 2\left[\left(\sum_{n=0}^{\infty}\Gamma (n,j)\sum_{k\geq 1}
\left[\frac{B_{n}(0,T^{\gamma })B_{n,k}(T)}{d_{T}}\right]^{2}\right)^{1/2}\right.\nonumber\\
&&\left.\hspace*{3cm}+\left(\sum_{n=0}^{\infty}\Gamma (n,j)\sum_{k\geq 1}\left[\widetilde{B}_{n}(0)\widetilde{B}_{k}\right]^{2}\right)^{1/2}\right]^{2} \nonumber\\
&&\leq  2(\sqrt{\mathcal{M}}+1)^{2}\left\|\mathcal{K}^{2}_{\alpha_{S},\alpha_{T}}\right\|_{L^{1}(L^{2}(\mathbb{S}_{d}(1)\times \mathbb{T}(1),d\nu\otimes dt))}<\infty.
\label{fourthmtrcbv}
\end{eqnarray}

\noindent From equation  (\ref{fourthmtrcbv}), Dominated Convergence Theorem   leads to  the convergence, in the mean-square sense, of $S_{T}$ to $S_{\infty}$ in  equation (\ref{mslimit}).
\end{proof}
\section{Final Comments}
\label{discussion}
The spatiotemporal  noncentral limit theorems derived in this paper characterize the asymptotic probability distribution of functionals of nonlinear transformations of Gaussian STRFs with Hermite rank equal to two. Under the same setting of conditions, similar results can be obtained  in the first Laguerre Chaos,
for chi--squared subordinated STRFs. Furthermore, in relation to the second Laguerre Chaos, Theorem 3 in \cite{Leonenkoetal17}  can be extended to  the context of LRD spatially homogeneous and isotropic, and stationary in time  chi--squared  STRFs subordinators,  by applying similar approaches to the ones adopted   in this paper (see also \cite{Caponera25} for the case of connected and compact two point homogeneous spaces  $\mathbb{M}_{d}$).
\subsection*{Acknowledgements}
This work has been partially supported by grants  \linebreak PID2022-142900NB-I00, funded by MICIU / AEI/10.13039/501100011033/ ERDF, EU, and  CEX2020-001105-M, funded by MICIU / AEI/10.13039/501100011033.

\bibliographystyle{amsplain}

\end{document}